\documentclass{amsart}
\usepackage{amsmath}
\usepackage{graphicx}
\usepackage{latexsym}
\usepackage{color}
\usepackage{amscd}
\usepackage[all]{xy}
\usepackage{enumerate}
\usepackage{soul}
\usepackage{comment}
\newtheorem{thm}{Theorem}
\newtheorem{lem}[thm]{Lemma}
\newtheorem{cor}[thm]{Corollary}

\newtheorem{lemma}[thm]{Lemma}

\theoremstyle{definition}

\newtheorem*{definition*}{Definition}
\newtheorem{dfn}[thm]{Definition}

\newtheorem{remark}[thm]{Remark}

\def\CM{\mathcal M}

\newcommand{\Z}{\mathbb{Z}}

\newcommand{\scl}{{\rm scl}}
\newcommand{\cl}{{\rm cl}}

\begin{document}

\title[On stable commutator lengths of Dehn twists  along separating curves]
{On stable commutator lengths of Dehn twists  along separating curves}

\author[N. Monden]{Naoyuki Monden}
\address{Department of Engineering Science, Osaka Electro-Communication University, Hatsu-cho 18-8, Neyagawa, 572-8530, Japan}
\email{monden@isc.osakac.ac.jp}

\author[K. Yoshihara]{Kazuya Yoshihara}
\address{}
\email{}

\begin{abstract}

We give new upper bounds on the stable commutator lengths of Dehn twists along separating curves in the mapping class group of a closed oriented surface. 
The estimates of these upper bounds are $O(1/g)$, where $g$ is the genus of the surface.
\end{abstract}

\maketitle

\setcounter{secnumdepth}{2}
\setcounter{section}{0}

\vspace{0.1in}
\section{Introduction}\label{section:one}

For $x$ in the commutator subgroup $[G,G]$ of a group $G$, we define the \textit{commutator length} $\cl_{G} (x)$ of $x$ to be the smallest number of commutators in $G$ whose product is equal to $x$. The \textit{stable commutator length} $\scl_{G} (x)$ of $x$ is the limit 
\[\scl_{G} (x) = \lim_{ n \to \infty} \frac{ \cl_{G} (x^n ) }{n}. \]

Let $\CM_g$ be the mapping class group of a closed oriented surface of genus $g$, and let $t_c$ be the Dehn twist along a simple closed curve $c$ on the surface. 
The natural problem is to calculate $\scl_{\CM_g}(t_c)$. 
However, in general, computing $\scl$ is difficult. 
Therefore, it makes sense to give upper and lower bounds on $\scl_{\CM_g}(t_c)$ (Note that using Louwsma's results \cite{l}, we find $\scl_{{\mathcal M}_{1}}(t_{c})=1/12$).

The lower bounds on $\scl_{\CM_g}(t_c)$ were given by Endo-Kotschick \cite{ek} using Gauge theory. 
For technical reasons, 
they showed that $1/(18g+6)\leq \scl_{\CM_g}(t_c)$ only for a separating curve $c$. This assumption is removed by Korkmaz \cite{ko}. 
He also gave upper bounds on stable commutator lengths of Dehn twists in \cite{ko}. 
An upper bound on $\scl_{\CM_g}(t_c)$ given in \cite{ko} (see also \cite{mo}) is constant for $g$. 
However, according to \cite{k} (see also \cite{ca}), there is an estimate $\scl_{\CM_g}(t_c) = O(1/g)$, that is, $\lim_{g \to \infty} \scl_{\CM_g}(t_c) = 0$. 
Such an upper bound were given in \cite{cms} for a nonseparating curve $c$, and the estimate is $\scl_{\CM_g}(t_c) \leq 1/(4g+6+(2/g))$ for $g\geq 1$. 
On the other hand, to the best of our knowledge, there is no such upper bound for a separating curve $s_h$ on $\Sigma_g$, where $s_h$ separates into two components with genera $h$ and $g-h$
For this reason, we focus on separating curves. 
We prove the following. 
\begin{thm}\label{th:one}
For some integers $k$, $h$ and $r$, we assume that genus $g = k h + r$ $(h=1,2,\ldots,\left[\frac{g}{2}\right], \ r = 0,1,\ldots ,h-1)$.
Then we have 
\[ \scl_{\CM_g}(t_{s_h}) \leq \dfrac{h(2h+1)(2g-2h+1)}{(g+1)(2g+1)-(2g-2h+1)r} \left(\dfrac{1}{2r+1}\scl_{\CM_g}(t_{s_r}) +1 \right). \]
\end{thm}
Note that $t_{s_0}=id$ in $\CM_g$ since $s_0$ bounds a disk, and therefore $\scl(t_{s_0})=0$. 
As a corollary, we obtain the following results. 
\begin{cor}\label{cor:one}
$\scl_{\CM_g}(t_{s_1}) \leq \dfrac{3(2g-1)}{(g+1)(2g+1)}$. 
\end{cor}
\begin{cor}\label{cor:two}
If $g=kh$, then $\scl_{\CM_g}(t_{s_h}) \leq \dfrac{h(2h+1)(2g-2h+1)}{(g+1)(2g+1)}$.
\end{cor}
Using Theorem~\ref{th:one} and Corollary~\ref{cor:one}, we can inductively give an upper bound of $\scl_{\CM_g}(t_{s_h})$ for any $h$.

Here is an outline of this paper. 
In Section~\ref{ss:one}, we review basic facts of stable commutator lengths and quasi-morphisms. 
Especially, we present ``Bavard's duality theorem" which gives a relationship between stable commutator lengths and quasi-morphisms. 
Section~\ref{ss:two} provides relations in mapping class groups. These relations are need to prove Theorem~\ref{th:one}. 
In the last section, we prove Theorem~\ref{th:one}. 
\begin{remark}
There are applications of giving lower and upper bounds on the stable commutator length of a Dehn twist. 
Powell \cite{po} showed that any element $x$ in $\CM_g$ can be factorized as a product of commutators if $g\geq 3$, that is, $\CM_g$ is a perfect group. 
The next problem is uniformly perfectness of $\CM_g$. 
Here, a group $G$ is called uniformly perfect if there is a natural number $N$ such that $\cl_G(x)\leq N$ for any element $x$ in $G$. 
It was conjectured by Morita \cite{m} that $\CM_g$ is not uniformly perfect, more strongly, the map $H^2_b(\CM_g)\to H^2(\CM_g;\mathbb{R})$ is not injective, where $H^2(\CM_g)$ (resp. $H^2_b(\CM_g)$) is the second (resp. bounded) cohomology of $\CM_g$. 
Endo-Kotschick \cite{ek} soloved them affirmatively by giving a lower bound of $\scl_{\CM_g}(t_c)$. 
Recently, the first author \cite{mo2} improved an upper bound on $\scl_{\CM_g}(t_c)$ for any simple closed curve $c$ and $g\geq 3$ by giving an explicit factorization of $t_c^n$ for any $n$. 
Using this result, he improved an upper bound on the limit of $\lim_{n\to \infty}\frac{h_g(n)}{n}$ for odd $g$, where $h_g(n)$ is the minimal $h$ such that there exists a $\Sigma_g$-bundle over $\Sigma_h$ with signature $4n$. 
Computing the limit is open problem in Problem 2.18 (B) of \cite{Kirby}. 
\end{remark}
\begin{remark} 
We introduce some background results on factorizations of some power of a Dehn twist as a product of commutators. 
Korkmaz and Ozbagci \cite{ko1} proved that $\cl_{\CM_g}(t_c)=2$ for any simple closed curve $c$ if $g\geq 3$. 
It was shown that $\cl_{\CM_g}(t_d^{10})\leq 2$ for a nonseparating curve $d$ if $g\geq 2$ (see \cite{ko1}) and that $\cl_{\CM_2}(t_{s_1}^5)\leq 6$ for a separating curve $s_1$ (see \cite{ks}). 
For the mapping class group $\CM_g^1$ of a compact oriented surface of genus $g\geq 2$ with a boundary curve $\delta$, it was proved in \cite{bkm} proved that $\cl_{\CM_g^1}(t_\delta^n)=[n/2]$ for any $n$, so $\scl_{\CM_g^1}(t_\delta)=1/2$ (see also \cite{baykur}). 
In the above results, explicit factorizations of some power of a Dehn twist were given. 
\end{remark}
\begin{remark}
We explain a geometrical interpretation of (stable) commutator lengths in $\CM_g$ as Lefschetz fibrations, which play an important role in 4-dimensional topology. 
Note that by the works of \cite{GS} and \cite{Do}, a 4-manifold is symplectic if and only if it admits a Lefschetz fibration, up to blow-up. 
Let $x$ be a product of right-handed Dehn twists $t_{v_1} \cdots t_{v_n}$. If we give a relation $x=\prod_{i=1}^k [a_i,b_i]$ in $\CM_g$, then we obtain a Lefschetz fibration with fiber $\Sigma_g$ over $\Sigma_k$ with $n$ singular fibers such that each of them is obtained by collapsing simple closed curve $v_j$ on $\Sigma_g$ to a point (The Euler characteristic of the 4-manifold admitting this Lefschetz fibration is equal to $4(g-1)(k-1)+n$). 
Conversely, given a Lefschetz fibration, we get the above relation. 
From this, computing $\cl_{\CM_g}(x)$ means that calculating the minimal base genus of such Lefschetz fibration and the minimal Euler characteristic of its total space. 
Therefore, $\cl_{\CM_g}(t_c^n)$ gives the ``smallest" Lefschetz fibraitons (in the sense of the Euler characteristic) in the ``simplest" ones (in the sense of singular fibers). Thus, we can regard $\scl_{\CM_g}(t_c)$ as the growth rate of the genus of their bases. 
\end{remark}

\vspace{0.1in}
\section{ Preliminaries}\label{section:two}


\subsection{Stable commutator lengths and quasi-morphisms}\label{ss:one}
For a group $G$, let $[ G,G ]$ denote the commutator subgroup of $G$. 
\begin{dfn}\rm
For $x \in [ G,G ]$, we define the \textit{commutator length} $\cl_{G} (x)$ of $x$ to be the smallest number of commutators in $G$ whose product is equal to $x$. The \textit{stable commutator length} $\scl_{G} (x)$ of $x$ is the limit 
\[\scl_{G} (x) = \lim_{ n \to \infty} \frac{ \cl_{G} (x^n ) }{n}. \]
By convention we define $\cl_{G} (x) = \infty$ if $x \in G$ is not in $[ G,G ]$, and $\scl_G(x)<\infty$ if and only if $x^m \in [ G,G ]$ for some integer $m$. 
\end{dfn}
The limit exists since the non-negative function $n \to \cl_{G} (x^n)$ is subadditive, and $\cl_G$ and $\scl_G$ are class functions. 
We need the theory of quasi-morphisms in order to prove Theorem~\ref{th:one}. 
In this section, we recall the definition and basic properties of quasi-morphisms. We refer the reader to \cite{ca} for the details. 
\begin{dfn}\rm
\label{define:three}
Let $G$ be a group. A $\emph{quasi-morphism}$ on $G$ is a function $\phi : G \rightarrow \mathbb{R}$ for which there is a constant $D(\phi) \geq 0$ such that 
\[ | \phi (xy) - \phi (x) - \phi (y) | \leq D( \phi ) \]
for all $x,y \in G$. We call $D ( \phi )$ a \textit{defect} of $\phi$. A quasi-morphism $\phi$ is called \textit{homogeneous} if \[\phi(x^n) = n\phi(x)\] for all $x \in G$ and all $n \in \Z$. 
\end{dfn}

We recall the following basic properties of homogeneous quasi-morphisms (for example, see Section 5.5.2 of \cite{ca} and Lemma 2.1 (1) of \cite{k} for proofs). 
\begin{lem}
\label{lemma:three}
Let $\phi : G \to \mathbb{R}$ be a homogeneous quasi-morphism on a group $G$. For all $x,y \in G$, the following holds. 
\begin{enumerate}
\item[{\rm (a)}] $\phi (x ) = \phi (yxy^{-1} )$,

\item[{\rm (b)}] If $xy = yx$, then $ \phi (xy) = \phi (x) + \phi (y).$
\end{enumerate}
\end{lem}

\begin{thm}[{Bavard's Duality Theorem ~\cite{ba} } ]
\label{theorem:two}
Let $Q$ be the set of all homogeneous quasi-morphisms on a group $G$ with positive defects. For any $x \in [G,G]$, we have
\[ \scl_{G} (x) = \sup_{\phi \in Q , D(\phi) }  \frac{| \phi (x) |}{2 D(\phi) }.  \]
\end{thm}

\subsection{Relations in mapping class groups}\label{ss:two}

Let $\Sigma_g$ be closed connected oriented surface of genus $g\geq 2$, and let $\CM_g$ be the \textit{mapping class group} of $\Sigma_g$, that is the group of isotopy classes of orientation preserving self-diffeomorphisms of $\Sigma_g$. 
Since $\CM_g / [ \CM_g , \CM_g] $ is isomorphic to $\Z_{12}$ if $g=1$, $\Z_{10}$ if $g = 2$ (see \cite{bh}), and is trivial if $g \geq 3$ (see \cite{po}), we can define $\scl_{\CM_g} (x)$ for any $x \in \CM_g$.

Let $t_c$ be the right-handed Dehn twist along a simple closed curve $c$ on $\Sigma_g$. 
We review some relations among Dehn twists. 
More details can be found in \cite{fm}.

\begin{lem}\label{lemma:four}
For any two separating curves $s_h$ and $s_h^\prime$ on $\Sigma_g$ which separate into two components with genera $h$ and $g-h$, $t_{s_h}$ is conjugate to $t_{s_h^\prime}$ and $t_{s_{g-h}}$. 
\end{lem}

\begin{lem}\label{lemma:five}
If $c$ and $d$ are two disjoint simple closed curves on $\Sigma_g$, then $t_c t_d = t_d t_c$. 
\end{lem}

\begin{lemma}[The hyperelliptic involution]\label{lemma:six}
Let $c_{1},c_{2},\ldots, c_{2g+1}$ be nonseparating curves in $\Sigma_{g}$ such that $c_i$ and  $c_j$ are disjoint if $|i-j|\geq 2$, and that $c_i$ and $c_{i+1}$ intersect at one point. Then, the product 
\[\iota:=  t_{c_{1}}t_{c_{2}}\cdots t_{c_{2g}}t_{c_{2g+1}}^2 t_{c_{2g}} \cdots t_{c_{2}} t_{c_{1}}. \]
is the \textit{hyperelliptic involution}. In particular, we have the relation $\iota t_{c_i} = t_{c_i} \iota$ for $i=1,2,\ldots,2g+1$. 
\end{lemma}

\begin{lem}[The even chain relation]\label{lemma:seven}
For a positive integer $h$, let us consider a sequence of simple closed curves $c_1, c_2, \ldots, c_{2h}$ in $\Sigma_g$ such that $c_i$ and  $c_j$ are disjoint if $|i-j|\geq 2$, and that $c_i$ and $c_{i+1}$ intersect at one point. Then, a regular neighborhood of $c_1\cup c_2\cup \cdots \cup  c_{2h}$ is a subsurface of genus $h$ with connected boundary, denoted by $d_h$. We then have 
\[(t_{c_1} t_{c_2} \cdots t_{c_{2h}} )^{4h+2}=t_{d_h}.\]
\end{lem}

\section{The proof of Theorem\ref{th:one}}\label{ss:three}
In this section, we prove Theorem~\ref{th:one}. For this, we need the following Lemma. 
\begin{lem}\label{lemma:eight}
Let $G$ be a group, and let $x_1,x_2,\ldots,x_n$ be elements in $G$ such that $x_i x_j = x_j x_i$ for $|i-j|\geq 2$. Then, $x_1 x_3 \cdots x_{2m-1} \cdot x_2 x_4 \cdots x_{2m}$ (resp. $x_1 x_3 \cdots x_{2m+1} \cdot x_2 x_4 \cdots x_{2m}$) is conjugate to $x_1 x_2 \cdots x_{2m}$ if $n=2m$ (resp. $n=2m+1$). 
\end{lem}

\begin{proof}
We consider the case where $n=2m+1$. Set 
\begin{align*}
&X_{2k+1} = x_1 x_2 \cdots  x_{2k+1} \cdot x_{2k+3} x_{2k+5} \cdots x_{2m+1} \cdot x_{2k+2} x_{2k+4} \cdots x_{2m}, \\
&X_{2k} = x_1 x_2 \cdots x_{2k} \cdot x_{2k+2} x_{2k+4} \cdots x_{2m} \cdot x_{2k+1} x_{2k+3} \cdots x_{2m+1}. 
\end{align*}
Note that $X_1=x_1 x_3 \cdots x_{2m+1} \cdot x_2 x_4 \cdots x_{2m}$ and $X_{2m+1} = x_1 x_2 \cdots x_{2m+1}$. Therefore, in order to prove Lemma~\ref{lemma:eight}, it is sufficient to show that $X_{i+1}$ is conjugate to $X_i$. From the assumption, we have 
\begin{align*}
&(x_{2k+2} x_{2k+4} \cdots x_{2m}) \cdot X_{2k+1} \cdot (x_{2k+2} x_{2k+4} \cdots x_{2m})^{-1} \\
&= x_{2k+2} x_{2k+4} \cdots x_{2m} \cdot  x_1 x_2 \cdots x_{2k-1} x_{2k} x_{2k+1} \cdot x_{2k+3} x_{2k+5} \cdots x_{2m+1} \\
&= x_1 x_2 \cdots x_{2k-1} \cdot x_{2k} x_{2k+2} \cdots x_{2m} \cdot x_{2k+1} x_{2k+3} x_{2k+5} \cdots x_{2m+1} = X_{2k}.
\end{align*}
Therefore, $X_{2k+1}$ is conjugate to $X_{2k}$. By a similar argument, we can show that $X_{2k}$ is conjugate to $X_{2k-1}$. The proof for $n=2m$ is similar. This finishes the proof of Lemma~\ref{lemma:eight}. 
\end{proof}

We now give the proof of Theorem~\ref{th:one}. 
\begin{proof}[The proof of Theorem~\ref{th:one}]
Let $\phi$ be a homogeneous quasi-morphism on $\CM_g$, and let $c_1,c_2,\ldots,c_{2g+1}$ the simple closed curves in Lemma~\ref{lemma:six}. 
 For simplicity of notation, we write $t_i$ instead of $t_{c_i}$. For $i=2,3,\ldots,k-1$, we write 
\begin{align*}
&T_1 := t_1^2 t_2 t_3 \cdots t_{2h},& &T_i := t_{2(i-1)h+1} t_{2(i-1)h+2} \cdots t_{2ih},& \\
&T_k := t_{2(k-1)h+1} t_{2(k-1)h+2} \cdots t_{2(g-h)},& &T_{k+1} := t_{2(g-h)+1} t_{2(g-h)+2} \cdots t_{2g},& \\
&T_{k+2} := t_{2g+1}^2 t_{2g} t_{2g-1} \cdots t_{2(g-h)+2},& &S := t_{2(g-h)+1} t_{2(g-h)} \cdots t_2.& 
\end{align*}
Then, by Lemma~\ref{lemma:three} (a),~\ref{lemma:four} and~\ref{lemma:seven}, 
\begin{align}
&\phi(T_1) = \phi(T_{k+2}) = \frac{1}{4h}\phi(t_{s_h}), \label{eq1} \\
&\phi(T_i) = \phi(T_{k+1}) = \frac{1}{4h+2}\phi(t_{s_h}), \label{eq2} \\
&\phi(T_k) = \frac{1}{4r+2}\phi(t_{s_r}), \label{eq3} \\
&\phi(S) = \frac{1}{4(g-h)+2}\phi(t_{s_h}) \label{eq4}
\end{align}
for $i=2,3,\ldots,k-1$. From the conjugate $t_1 \iota t_1^{-1}$ and Lemma~\ref{lemma:six}, we have 
\[t_1^2 t_2 t_3 \cdots t_{2g} t_{2g+1}^2 t_{2g} \cdots t_3 t_2 = \iota^{-1} = \iota,\]
that is, $T_1 T_2 \cdots T_{k+2} S = \iota$. In particular, we see that the relation
\begin{align}
T_1 T_2 \cdots T_{k+2} = \iota \cdot S^{-1} \label{eq5}
\end{align}
holds in $\CM_g$. Since $\phi(\iota)=0$ by the definition of homogeneous quasimorphisms, from the equations (\ref{eq4}), (\ref{eq5}), Lemma~\ref{lemma:three} (b) and~\ref{lemma:six}, we obtain 
\begin{align}
\phi(T_1 T_2 \cdots T_{k+2}) = -\phi(S) = -\frac{1}{4(g-h)+2}\phi(t_{s_h}). \label{eq6}
\end{align}
Note that $T_i T_j = T_j T_i$ if $|i-j|\geq 2$ from Lemma~\ref{lemma:five}. From this and Lemma~\ref{lemma:eight}, we have
\begin{align*}
\phi(T_1 T_2 \cdots T_{k+2}) = 
  \left\{ \begin{array}{ll}
      \displaystyle \phi (T_1 T_3 \cdots T_{k+1} \cdot T_2 T_4 \cdots T_{k+2} ) & \ \ \mathrm{if} \ k+2 \ \mathrm{is \ even} \\[1pt]
      \displaystyle \phi (T_1 T_3 \cdots T_{k+2} \cdot T_2 T_4 \cdots T_{k+1} ) & \ \ \mathrm{if} \ k+2 \ \mathrm{is \ odd}.
      \end{array} \right. 
\end{align*}
By this equation, the definition of quasimorphisms, Lemma~\ref{lemma:three} (b) and the equations (\ref{eq1})-(\ref{eq3}) and (\ref{eq6}), 
\begin{align*}
D(\phi) & \geq |\phi(T_1 T_2 \cdots T_{k+2}) - \sum_{i=1}^{k+2} \phi(T_i)| \\
&= \left|-\frac{1}{4(g-h)+2}\phi(t_{s_h}) - \dfrac{2}{4h}\phi(t_{s_h}) - \dfrac{k-1}{4h+2}\phi(t_{s_h})  - \dfrac{1}{4r+2}\phi(t_{s_r}) \right|
\end{align*}
This gives 
\begin{align*}
\dfrac{(g+1)(2g+1)-(2g-2h+1)r}{2h(2h+1)(2g-2h+1)} \left| \phi(t_{s_h}) \right| \leq \dfrac{1}{4r+2} \left| \phi(t_{s_r}) \right| + D(\phi) 
\end{align*}
from $r=g-kh$. Hence, we obtain 
\begin{align*}
\dfrac{\left| \phi(t_{s_h}) \right|}{2D(\phi)} \leq \dfrac{h(2h+1)(2g-2h+1)}{(g+1)(2g+1)-(2g-2h+1)r} \left( \dfrac{1}{2r+1} \dfrac{\left| \phi(t_{s_r}) \right|}{2D(\phi)} + 1 \right). 
\end{align*}
By Bavard's Duality Theorem, 
\begin{align*}
\scl_{\CM_g}(t_{s_h}) \leq \dfrac{h(2h+1)(2g-2h+1)}{(g+1)(2g+1)-(2g-2h+1)r} \left(\dfrac{1}{2r+1}\scl_{\CM_g}(t_{s_r}) +1 \right). 
\end{align*}
This completes the proof of Theorem~\ref{th:one}. 
\end{proof}

\vspace{0.1in}
\noindent \textit{Acknowledgements.} 
The first author was supported by Grant-in-Aid for Young Scientists (B) (No. 13276356), Japan Society for the Promotion of Science. 
The second author would like to thank Susumu Hirose, Noriyuki Hamada, and Takayuki Okuda for helpful comments and invaluable advice on the mapping class groups of the surfaces. Finally, he also would like to thank Professor Osamu Saeki for many helpful suggestions and comments.

\vspace{0.3in}

\end{document}